\documentclass[reqno]{amsart}
\usepackage[utf8x]{inputenc}  

\usepackage{setspace}
\usepackage[left=3.1cm, right=3.1cm, bottom=4cm]{geometry}                 

\usepackage{graphicx} 
\usepackage{pgf,tikz}
\usetikzlibrary{arrows}
\usepackage{amssymb}
\usepackage{pdfsync}
\usepackage{mathrsfs}
\usepackage{hyperref} 

\usepackage{epstopdf}
\usepackage{bbm} 
\usepackage[colorinlistoftodos,prependcaption,textsize=tiny]{todonotes}
\usepackage{xargs}
\usepackage{bm}
\usepackage{lmodern}

\def\R{\mathbb{R}}
\def\T{\mathbb{T}}

\def\Z{\mathbb{Z}}
\def\C{\mathbb{C}}

\def\Q{\mathbb{Q}}

\def\E{\mathcal{E}}

\renewcommand{\d}{\text{\rm d}}

\newcommand{\bxi}{{\bm \xi}}
\newcommand{\balpha}{{\bm \alpha}}

\newcommand{\mc}{\mathcal}

\newtheorem{theorem}{Theorem}

\newtheorem{corollary}[theorem]{Corollary}

\newtheorem*{definition*}{Definition}

\newtheorem{lemma}[theorem]{Lemma}

\makeatletter
\DeclareFontFamily{U}{tipa}{}
\DeclareFontShape{U}{tipa}{m}{n}{<->tipa10}{}
\newcommand{\arc@char}{{\usefont{U}{tipa}{m}{n}\symbol{62}}}%

\makeatother

\numberwithin{equation}{section}

\allowdisplaybreaks

\newcommand{\intav}[1]{\mathchoice {\mathop{\vrule width 6pt height 3 pt depth  -2.5pt
\kern -8pt \intop}\nolimits_{\kern -6pt#1}} {\mathop{\vrule width
5pt height 3  pt depth -2.6pt \kern -6pt \intop}\nolimits_{#1}}
{\mathop{\vrule width 5pt height 3 pt depth -2.6pt \kern -6pt
\intop}\nolimits_{#1}} {\mathop{\vrule width 5pt height 3 pt depth
-2.6pt \kern -6pt \intop}\nolimits_{#1}}}

\newcommand{\intavl}[1]{\mathchoice {\mathop{\vrule width 6pt height 3 pt depth  -2.5pt
\kern -8pt \intop}\limits_{\kern -6pt#1}} {\mathop{\vrule width 5pt
height 3  pt depth -2.6pt \kern -6pt \intop}\nolimits_{#1}}
{\mathop{\vrule width 5pt height 3 pt depth -2.6pt \kern -6pt
\intop}\nolimits_{#1}} {\mathop{\vrule width 5pt height 3 pt depth
-2.6pt \kern -6pt \intop}\nolimits_{#1}}}

  \newcommand{\z}{{\bm z}}
    
     \newcommand{\n}{{\bm n}}
      \newcommand{\s}{{\bm s}}      
      \newcommand{\0}{{\bm 0}}

        \renewcommand{\t}{{\bm t}}
         
    \renewcommand{\bxi}{{\bm \xi}}   
       \newcommand{\btheta}{{\bm \theta}}

\title[Effective equidistribution of Galois orbits]{Effective equidistribution of Galois orbits\\ for mildly regular test functions}

\author[Carneiro]{Emanuel Carneiro}
\address{ICTP - The Abdus Salam International Centre for Theoretical Physics, 
Strada Costiera, 11, I - 34151, Trieste, Italy.}
\email{carneiro@ictp.it}

\author[Das]{Mithun Kumar Das}
\address{ School of Mathematical Sciences, National Institute of Science Education and Research, A CI of 
		Homi Bhabha National Institute, Jatni, Khurda, 
		752050, India.}
\email{das.mithun3@gmail.com}

\date{\today}                                           
\begin{document}

\begin{abstract} In this paper we provide a detailed study on effective versions of the celebrated Bilu's equidistribution theorem for Galois orbits of sequences of points of small height in the $N$-dimensional algebraic torus, identifying the quantitative dependence of the convergence in terms of the regularity of the test functions considered. We develop a general Fourier analysis framework that extends previous results obtained by Petsche (2005), and by D'Andrea, Narv\'aez-Clauss and Sombra (2017).
\end{abstract}

\maketitle 

\section{Introduction}
This is a paper in analysis motivated by a problem in algebraic number theory. Let $\overline{\Q}$ denote the algebraic closure of $\Q$. By $\overline{\mathbb{Q}}^{\times}$ and $\C^{\times}$ we denote the multiplicative groups of $\overline{\Q}$ and $\C$, respectively. For $N \geq 1$, Bilu's equidistribution theorem \cite{Bilu} establishes that the Galois orbits of strict sequences of points of small Weil height in $\big(\overline{\mathbb{Q}}^{\times}\big)^N$ tend to the uniform distribution around the unit polycircle $(\mathbb{S}^1)^N$. Our aim here is to provide a detailed study on effective versions of this theorem via Fourier analysis, identifying the quantitative dependence on the regularity of the test functions considered.

\smallskip

There are important previous works in the literature that touch on the theme of effective estimates. In dimension $N=1$ we can cite, for instance, the works of Petsche \cite{P}, Favre and Rivera-Letelier \cite{FR}, Pritsker \cite{Pr}, Baker and Masser \cite{BM}, and Amoroso and Plessis \cite{AP}. In the multidimensional case $N\geq 1$, the elegant work of D'Andrea, Narv\'aez-Clauss and Sombra \cite{DNS} is our main reference on the theme. The effective estimates in these works have a common underlying characteristic: they assume a reasonably good level of regularity on the test functions, usually a suitable notion of Lipschitz continuity. The test functions on Bilu's equidistribution theorem are, in principle, only bounded and continuous. From the analysis point of view, there is a vast world between continuous and Lipschitz continuous functions, and the main motivation for  this paper is to close this gap in a systematic way. One can consider, for instance, the concepts of H\"{o}lder continuity and fractional derivatives, and our main results (Theorems \ref{Thm1} and \ref{Thm5}, and Corollaries \ref{Cor3} and \ref{Cor6}) provide the sharp quantitative dependence of the convergence in terms of the height in these cases. As far as applications go, in Appendix \ref{App_A} we show how our methods yield a new bound for the multidimensional angular discrepancy and, in Appendix \ref{App_C}, we briefly complete the proof of Bilu's equidistribution theorem using our effective estimates for mildly regular test functions.

\smallskip 

Our study treats the multidimensional situation, with Fourier analysis playing an important role. We extend the setup of D'Andrea, Narv\'aez-Clauss and Sombra \cite{DNS}, and of Petsche \cite{P}, to a broader Fourier analytic framework, providing refinements over the cases considered in these papers. The one-dimensional effective estimates of the papers \cite{AP, BM, FR, Pr} consider a slightly different setup, bringing certain tools from potential theory to control the contribution of the angular part instead of Fourier analysis; hence our work, when restricted to dimension $N=1$, complements these but does not directly compare to them (in the sense that one can find examples of test functions that are included in our framework but are not included in the frameworks of these papers, and vice-versa).

\subsection{Background} We start by reviewing the basic terminology in order to state our results. We generally follow the notation of \cite{DNS, P} to facilitate the references. 

\smallskip

For a finite set $S \subset (\C^{\times})^N$, define the discrete probability measure on $(\C^{\times})^N$ associated to it by 
$$\mu_{S} := \frac{1}{|S|}\sum_{\z \in S} \delta_{ \z}\,,$$
where $|S|$ denotes the cardinality of the set $S$ and $\delta_{\z}$ denotes the Dirac delta measure on $(\C^{\times})^N$ supported in $\z$. Let $(\mathbb{S}^1)^N = \{\z = (z_1, z_2, \ldots, z_N) \in \C^N;\, |z_1| = |z_2| = \ldots = |z_N| = 1\}$ be the unit polycircle in $\C^N$. Note that $(\mathbb{S}^1)^N$ is a compact subgroup of $\C^N$ that is isomorphic to $(\R/\Z)^N$. We let $\mu_{(\mathbb{S}^1)^N}$ be the Haar probability measure on $(\mathbb{S}^1)^N$, considered as a measure on $(\C^{\times})^N$.

\smallskip

For $P(x) = \sum_{j=0}^d c_jx^j = c_d \prod_{j=1}^d (x - \alpha_j)$ a non-zero polynomial in $\C[x]$, its logarithmic Mahler measure is 
$$m(P) := \int_{\C^{\times}} \log \left|P(z)\right| 
\d \mu_{\mathbb{S}^1}(z) = \int_{\R/\Z} \log\big| P\big(e^{2\pi i \theta}\big)\big|\,\d\theta =  \log |c_d| + \sum_{j=1}^d \log^+ |\alpha_j|\,,$$
where $\log^+x:= \max\{0, \log x\}$, for $x >0$. If $\xi \in \overline{\Q}$, its Weil height is defined as
$$h(\xi) := \frac{m(P_{\xi})}{\deg(\xi)}\,,$$
where $P_{\xi} \in \Z[x]$ is the minimal polynomial of $\xi$ over $\Q$, and $\deg(\xi) := [\Q(\xi):\Q] = \deg(P_{\xi})$ is the degree of $\xi$. We extend this notion of height to $\big(\overline{\mathbb{Q}}^{\times}\big)^N$ by setting, for $\bxi = (\xi_1, \xi_2, \ldots, \xi_N) \in \big(\overline{\mathbb{Q}}^{\times}\big)^N$, 
$$h(\bxi) := h(\xi_1) + h(\xi_2) + \ldots + h(\xi_N).$$
If $\bxi = (\xi_1, \xi_2, \ldots, \xi_N) \in \big(\overline{\mathbb{Q}}^{\times}\big)^N$, its Galois orbit $S$ is the finite set 
$$S = \big\{(\sigma \xi_1, \sigma \xi_2, \ldots, \sigma \xi_N) \,: \, \sigma \in {\rm Gal}\big(\overline{\Q}/\Q\big)\big\}.$$

\smallskip

For $\bxi \in \big(\overline{\mathbb{Q}}^{\times}\big)^N$, and $F:(\C^{\times})^N \to \C$ a test function, we henceforth define the quantity
\begin{equation*}
\E(F, \bxi):= \left| \int_{(\C^{\times})^N} F \,\d \mu_{S} - \int_{(\C^{\times})^N} F \,\d \mu_{(\mathbb{S}^1)^N} \right| \,,
\end{equation*}
where $S$ is the Galois orbit of $\bxi$. 

\smallskip

We say that a sequence $\{\bxi_k\}_{k \geq 1} \subset \big(\overline{\mathbb{Q}}^{\times}\big)^N$ is strict if any proper algebraic subgroup of $\big(\overline{\mathbb{Q}}^{\times}\big)^N$ contains $\bxi_k$ for only finitely many values of $k$. In 1997, Bilu \cite{Bilu} proved the following result.

\begin{theorem}[Bilu] \label{BiluThm}Let $\{\bxi_k\}_{k \geq 1} \subset \big(\overline{\mathbb{Q}}^{\times}\big)^N$ be a strict sequence with $\lim_{k \to \infty} h(\bxi_k) = 0$. Then, for any bounded and continuous function $F:(\C^{\times})^N \to \C$,
\begin{align}\label{20240923_16:19}
\lim_{k \to \infty} \E(F, \bxi_k) = 0.
\end{align}
\end{theorem}
This result was inspired by previous equidistribution results for points of small height on abelian varieties by Szpiro, Ullmo and Zhang \cite{SUZ, Z}, and the subject flourished with many generalizations to other notions of heights and places over the past decades; see, e.g., \cite{BH, BR, BPRS, C, CT, Gu, K, R, Y}. As already mentioned, our goal here is to provide effective bounds for the convergence \eqref{20240923_16:19} in terms of the regularity of the test function $F$, in the spirit of the previous works \cite{DNS, P}. 

\smallskip

For $\n  = (n_1, n_2, \ldots, n_N) \in \Z^N$ we consider the map $\chi^{\n}:\big(\overline{\mathbb{Q}}^{\times}\big)^N \to \overline{\mathbb{Q}}^{\times}$ given by 
$$\bxi = (\xi_1, \xi_2, \ldots ,\xi_N) \mapsto \chi^{\n}(\bxi) := \xi_1^{n_1}\xi_2^{n_2}\ldots \xi_N^{n_N}. $$
Following \cite{DNS}, we define the generalized degree of a point $\bxi \in \big(\overline{\mathbb{Q}}^{\times}\big)^N$ by 
\begin{equation}\label{20241104_14:58}
\mathcal{D}(\bxi) := \min_{\n \neq {\bm 0}} \big\{ \|\n\|_1  \deg\big(\chi^{\n}(\bxi)\big)\big\}\,,
\end{equation}
where $\|\n\|_1 = |n_1| + \ldots + |n_N|$. We observe that when $N=1$ the two notions of degree coincide, i.e., $\mathcal{D}(\xi) = \deg(\xi)$. Moreover, the Northcott property implies that a strict sequence $\{\bxi_k\}_{k \geq 1}$ with $\lim_{k\to \infty} h(\bxi_k) = 0$ must verify $\lim_{k\to \infty} \mathcal{D}(\bxi_k)  = \infty$; see Lemma \ref{Lem12} in Appendix \ref{App_B} for details. Throughout the paper, for $\bxi  \in \big(\overline{\mathbb{Q}}^{\times}\big)^N$, we let 
\begin{equation}\label{20241104_15:14}
h_{\mathcal{D}}(\bxi):= h(\bxi) + \frac{\log\big(2\mathcal{D}(\bxi)\big)}{3 \mathcal{D}(\bxi)}.
\end{equation}

\smallskip

Fourier analysis plays an important role in our study, as the regularity properties of a function $F:(\C^{\times})^N \to \C$ are related to the integrability properties of its Fourier transform. Letting $\T := \R/\Z$, we identify $\T^N \times \R^N$ with $(\C^{\times})^N$ via the logarithmic-polar change of coordinates
\begin{equation}\label{20251006_13:34}
(\btheta, \s) =\big( (\theta_1,  \ldots, \theta_N), (s_1, \ldots, s_N) \big) \mapsto \big(e^{2\pi i \theta_1 + s_1}, \ldots, e^{ 2\pi i \theta_N + s_N}\big) = (z_1, \ldots, z_N) = \z. 
\end{equation}
From now on, we regard our test functions as $F = F(\btheta, \s)$ with $(\btheta, \s) \in \T^N \times \R^N$ (it is equivalent to consider Fourier analysis on the locally compact abelian group $(\C^{\times})^N$, as in \cite{P}, or in $ \T^N \times \R^N$ via this change of coordinates; we adopt the latter). When $F \in L^1(\T^N \times \R^N)$, its classical Fourier transform  $\widehat{F}: \Z^N \times \R^N \to \C$ is defined by 
$$\widehat{F}(\n, \t):= \int_{\T^N} \int_{\R^N} F(\btheta, \s) \,e^{-2 \pi i \n \cdot \btheta} \, e^{-2 \pi i \t \cdot \s}\,\d \s \,\d \btheta.$$
By Plancherel's theorem, this operator extends to an isometry from $L^2(\T^N \times \R^N)$ to $L^2(\Z^N \times \R^N)$. If $\widehat{F} \in L^1(\Z^N \times \R^N)$, Fourier inversion holds and we have
\begin{equation}\label{20241025_10:59}
F(\btheta, \s) = \sum_{\n \in \Z^N} \int_{\R^N} \widehat{F}(\n, \t)\,e^{2 \pi i \n \cdot \btheta} \, e^{2 \pi i \t \cdot \s}\,\d \t.
\end{equation}
Given $F: \T^N \times \R^N\to \C$ continuous, throughout the paper we let $F_{\bf 0}: \T^N \to \C$ be the restriction $F_{\bf 0}(\btheta):= F(\btheta, {\bm 0})$. Its Fourier transform $\widehat{F_{\bf 0}}:\Z^N \to \C$ is given by
$$\widehat{F_{\bf 0}}(\n) = \int_{\T^N} F_{\bf 0}(\btheta) \,e^{-2 \pi i  \n \cdot \btheta} \,\d \btheta.$$
Although the notation for the Fourier transform is the same, it should be sufficiently clear from the context in which locally compact abelian group we are applying it (either $\T^N \times \R^N$ or $\T^N$).

\subsection{Effective equidistribution of Galois orbits} We now move on to describing our main results. We present here two points of view, depending on whether our regularity assumptions take place completely or partially on the Fourier space.  

\subsubsection{Regularity of $F$ on the Fourier space} Our first point of view considers the regularity of the test function $F$ given in terms of integrability conditions on the Fourier side. This extends the approach of Petsche \cite{P}. Let $\mc{A}$ be the following class of test functions:
\begin{equation*}
\mc{A} = \big\{ F: \T^N \times \R^N\to \C \ {\rm continuous};\ F \in L^2(\T^N \times \R^N);\ {\rm and} \ \widehat{F} \in L^1(\Z^N \times \R^N)\big\}.
\end{equation*}
Observe that if $F \in \mc{A}$, then $F_{\bf 0}: \T^N \to \C$ is continuous and  
\begin{align*}
\widehat{F_{\bf 0}}(\n) = \int_{\R^N} \widehat{F}(\n, \t)\,\d \t.
\end{align*}
Our first result is as follows.

\begin{theorem} \label{Thm1}
Let $G: (0,\infty) \to (0,\infty)$ and $H: (0,\infty) \to (0,\infty)$ be functions such that 
\begin{enumerate}
\item[(i)] $G$ and $H$ are non-decreasing and $\lim_{x \to \infty} G(x) =  \lim_{x \to \infty} H(x) = \infty$;
\smallskip
\item[(ii)] The functions $G(x)/\sqrt{x} $ and $H(x)/\sqrt{x}$ are non-increasing.
\smallskip
\end{enumerate}
Then, for any $\bxi \in \big(\overline{\mathbb{Q}}^{\times}\big)^N$ and $F \in \mc{A}$ we have
\begin{align}\label{20241026_14:46}
\begin{split}
\E(F, \bxi) & \leq \frac{2 \,C_1(F, G)}{G\big( (8 \pi  h(\bxi))^{-1}\big)} + \frac{C_2(F,H)}{H \big( (24 \,h_{\mathcal{D}} (\bxi))^{-1} \big)}\,,
\end{split}
\end{align}
where
\begin{align}\label{20250116_14:25}
C_1(F,G) := \sum_{\n \in \Z^N} \int_{\R^N} \left| \widehat{F}(\n, \t)\right|G( \|\t\|_{\infty})\, \d\t\,,
\end{align}
and
\begin{align}\label{20250116_14:26}
C_2(F,H) := \sum_{\n \in \Z^N\setminus \{\0\}} \left| \widehat{F_{\0}}(\n)\right|H( \|\n\|_{1}).
\end{align}
\end{theorem}
\noindent{\sc Remark:} The appearance of $\|\n\|_{1}$ in  \eqref{20250116_14:26} is due to our definition of the generalized degree $\mathcal{D}(\bxi)$ in \eqref{20241104_14:58}, chosen simply to facilitate the comparison with \cite{DNS}; see \eqref{20241025_16:22} below for the precise point where it appears in the proof. Had we defined the generalized degree slightly differently, by setting $\mathcal{D}_p(\bxi) := \min_{\n \neq {\bm 0}} \big\{ \|\n\|_p  \deg(\chi^{\n}(\bxi))\big\}$ for $1 \leq p\leq \infty$, where $\|\n\|_p = \big(\sum_{j=1}^N |n_j|^p\big)^{1/p}$ if $1 \leq p < \infty$ and $\|\n\|_{\infty} = \max_{j} |n_j|$, and replaced $\mathcal{D}(\bxi)$ by $\mathcal{D}_p(\bxi)$ in the definition \eqref{20241104_15:14}, our proof would work {\it ipsis litteris} yielding the same estimate as in \eqref{20241026_14:46}, with $\|\n\|_1$ replaced by $\|\n\|_p$ in \eqref{20250116_14:26}. The same remark applies to all of our results in this Introduction. 

\smallskip

Naturally, if the constants in \eqref{20250116_14:25} and \eqref{20250116_14:26} are not finite, estimate \eqref{20241026_14:46} is trivially true. For the theorem to be meaningful, it is expected that these constants be finite, which should be seen as the regularity assumptions on the test function $F$ (i.e., a bit more integrability than simply $ \widehat{F} \in L^1(\Z^N \times \R^N)$). For instance, for $0 < \gamma \leq 1/2$, we may consider $G(x) = H(x) = x^{\gamma}$. Using the fact that $h(\bxi) \leq h_{\mathcal{D}}(\bxi)$, Theorem \ref{Thm1} plainly yields the following corollary.

\begin{corollary} \label{Cor3} Let $0 < \gamma \leq 1/2$ be given. Then, for any $\bxi \in \big(\overline{\mathbb{Q}}^{\times}\big)^N$ and $F \in \mc{A}$ we have
\begin{align}\label{20241026_15:13}
\E(F, \bxi)  & \leq C(F) \, h_{\mathcal{D}}(\bxi)^{\gamma},
\end{align}
with 
\begin{equation}\label{20241026_15:38}
C(F) := \sum_{\n \in \Z^N} \int_{\R^N} \left| \widehat{F}(\n, \t)\right| \Big(2 (8\pi)^{\gamma} \|\t\|_{\infty}^{\gamma} + 24^{\gamma} \|\n\|_1^{\gamma}\Big)\,\d\t.
\end{equation}
Moreover, the exponent $\gamma$ in \eqref{20241026_15:13} is sharp in the following sense: there exists a function $F \in \mc{A}$ with \eqref{20241026_15:38} finite, and a sequence $\{\bxi_k\}_{k \geq 1} \subset \big(\overline{\mathbb{Q}}^{\times}\big)^N$ with $\lim_{k \to \infty} h_{\mathcal{D}}(\bxi_k) = 0$, for which
\begin{equation}\label{20250227_14:39}
\E(F, \bxi_k) \geq \frac{h_{\mathcal{D}}(\bxi_k)^{\gamma}}{|\log h_{\mathcal{D}}(\bxi_k)|\, \big(\log |\log h_{\mathcal{D}}(\bxi_k)|\big)^2 }.
\end{equation}
\end{corollary}
The sharpness of Corollary \ref{Cor3} is discussed in Section \ref{Qual_Sharp}. The term $h_{\mathcal{D}}(\bxi)$ is the quantity that goes to zero in Bilu's equidistribution theorem, therefore, when looking at an estimate like \eqref{20241026_15:13}, the higher the power of $h_{\mathcal{D}}(\bxi)$ the better the estimate, quantitatively speaking. For instance, the main result of \cite{P} (in dimension $N=1$) presents the upper bound 
\begin{align*}
\E(F, \xi)  \leq c \left( \sum_{n \in \Z} \int_{\R} \left| \widehat{F}(n, t)\right| \big(1 + |t| +  |n|\big)\,\d t\right)h_{\mathcal{D}}(\xi)^{1/3}\,,
\end{align*}
where $c>0$ is a constant. Corollary \ref{Cor3}, for instance with $\gamma = 1/2$, improves this estimate both in the exponent of $h_{\mathcal{D}}(\xi)$ and in the milder regularity assumption on $F$. 

\smallskip

From the analysis point of view, the assumption that \eqref{20241026_15:38} is finite can be thought of as $F$ having a fractional derivative of order $\gamma$. The exponent $1/2$ for $h_{\mathcal{D}}(\xi)$ is the natural limit of our method, and even if $F$ has a fractional derivative of order $a >1/2$, one would use the basic fact that $y^{1/2} \lesssim 1 + y^{a}$ to circle back to our result in Corollary \ref{Cor3} with $\gamma = 1/2$. It is also interesting to contemplate Theorem \ref{Thm1} with slowly increasing functions $G$ and $H$, for instance $G(x) = H(x) = \log(2+x)$. 

\smallskip

As suggested by Theorem \ref{Thm1}, some level of information on the regularity of $F$ is required in order to quantify Bilu's equidistribution. In our current setup, one may wonder whether the mere assumption that $\widehat{F} \in L^1(\Z^N \times \R^N)$ is enough (after all, by the Riemann-Lebesgue lemma, this implies the continuity of $F$). We answer this question affirmatively, in terms of the tail function
\begin{align*}
\nu_{\widehat{F}}(y) := \sum \int_{\|\n\|_1 + \|\t\|_1 >y} \left|\widehat{F}(\n, \t)\right| \d\t\,,
\end{align*}
that verifies $\lim_{y \to \infty} \nu_{\widehat{F}}(y) =0$.  

\begin{theorem}\label{Thm4}
Let $W: (0,\infty) \to (0,\infty)$ be a non-decreasing function such that $\lim_{x\to \infty} W(x) = \infty$ and $\lim_{x\to \infty} W(x)/x = 0$. Then, for any $\bxi \in \big(\overline{\mathbb{Q}}^{\times}\big)^N$ and $F \in \mc{A}$ we have 
\begin{align*}
\E(F, \bxi)  \leq 2\big(\sqrt{8\pi} + \sqrt{6}\big) \sqrt{h_{\mc{D}}(\bxi) \, W\big(h_{\mc{D}}(\bxi)^{-1}\big)}\, \big\|\widehat{F}\big\|_{L^1(\Z^N \times \R^N)}  \,+\, 3\,\nu_{\widehat{F}} \Big(W\big(h_{\mc{D}}(\bxi)^{-1}\big)\Big).
\end{align*}
\end{theorem}
As an application of this result, one may take, for instance, $W(x) = \sqrt{x}$ to get 
\begin{align*}
\E(F, \bxi) \leq 2\big(\sqrt{8\pi} + \sqrt{6}\big) h_{\mc{D}}(\bxi)^{1/4}\, \big\|\widehat{F}\big\|_{L^1(\Z^N \times \R^N)}  + 3\,\nu_{\widehat{F}} \big(h_{\mc{D}}(\bxi)^{-1/2}\big).
\end{align*}

\smallskip

\subsubsection{Angular regularity of $F$ on the Fourier space} Our second point of view considers the regularity of the log-radial part of $F$ on the physical space, and the regularity of the angular part of $F$ on the Fourier space. This extends the approach of D'Andrea, Narv\'aez-Clauss and Sombra \cite{DNS}.

\smallskip

Here we start simply with $F: \T^N \times \R^N\to \C$ continuous. Let $\omega:[0,\infty) \to [0,\infty)$ be a non-decreasing function such that $\lim_{x \to 0} \omega(x) = 0$. We say that $F$ admits a uniform modulus of continuity $\omega$ at $\s =\0$ if, for each $\btheta \in \T^N$, we have 
\begin{align}\label{20241028_10:36}
\big| F(\btheta, \s)  - F(\btheta, \0) \big| \leq \omega(|\s|). 
\end{align}
For instance, if $\omega(x) = cx$, condition \eqref{20241028_10:36} indicates a Lipschitz regularity at the origin, whereas if $\omega(x) = cx^{\gamma}$, with $0 < \gamma < 1$, condition \eqref{20241028_10:36} indicates a H\"{o}lder regularity at the origin. Our next result is the following.

\begin{theorem}\label{Thm5}
Let $\omega:[0,\infty) \to [0,\infty)$ be a non-decreasing and concave function such that $\lim_{x \to 0} \omega(x) = 0$, and assume that $F: \T^N \times \R^N\to \C$ is continuous and verifies \eqref{20241028_10:36}. Let $H: (0,\infty) \to (0,\infty)$ be a non-decreasing function such that $\lim_{x \to \infty} H(x) = \infty$ and $H(x)/\sqrt{x}$ is non-increasing. Then, for any $\bxi \in \big(\overline{\mathbb{Q}}^{\times}\big)^N$ we have
\begin{align}\label{20241028_12:03}
\begin{split}
\E(F, \bxi) \leq \omega\big(2 h (\bxi)\big) + \frac{C_2(F,H)}{H \big( (24 \,h_{\mathcal{D}} (\bxi))^{-1} \big)}\,,
\end{split}
\end{align}
where
\begin{align}\label{20250116_14:27}
C_2(F,H) := \sum_{\n \in \Z^N\setminus \{\0\}} \left| \widehat{F_{\0}}(\n)\right|H( \|\n\|_{1}).
\end{align}
\end{theorem}

Again, if the constant in \eqref{20250116_14:27} is not finite, estimate \eqref{20241028_12:03} is trivially true. For the theorem to be meaningful, it is expected that this constant be finite, which should be seen as the angular regularity assumption on the test function $F$. For instance, for $0 < \gamma \leq 1/2$, we may define the H\"{o}lder constant 
\begin{equation*}
L_{\gamma}(F) := \sup_{\substack{(\btheta, \s) \in \T^N \times \R^N\\ \s \neq \0}} \frac{\big| F(\btheta, \s)  - F(\btheta, \0) \big|}{|\s|^{\gamma}}.
\end{equation*}
If $L_{\gamma}(F) < \infty$, then \eqref{20241028_10:36} holds with $\omega(x) = L_{\gamma}(F)\, x^{\gamma}$. If we also consider $H(x) = x^{\gamma}$, Theorem \ref{Thm5} yields the following corollary (recall that $h(\bxi) \leq h_{\mathcal{D}}(\bxi)$).

\begin{corollary} \label{Cor6} Let $F: \T^N \times \R^N\to \C$ be a continuous function and let $0 < \gamma \leq 1/2$. Then, for any $\bxi \in \big(\overline{\mathbb{Q}}^{\times}\big)^N$ we have 
\begin{align}\label{20241029_15:48}
\E(F, \bxi)  & \leq C(F) \, h_{\mathcal{D}}(\bxi)^{\gamma},
\end{align}
with 
\begin{equation}\label{20241029_15:47}
C(F) := 2^{\gamma}L_{\gamma}(F)  + 24^{\gamma} \sum_{\n \in \Z^N }\left| \widehat{F_{\0}}(\n)\right|   \|\n\|_1^{\gamma}.
\end{equation}
Moreover, the exponent $\gamma$ in \eqref{20241029_15:48} is sharp in the following sense: there exists a continuous function $F$ with \eqref{20241029_15:47} finite, and a sequence $\{\bxi_k\}_{k \geq 1} \subset \big(\overline{\mathbb{Q}}^{\times}\big)^N$ with $\lim_{k \to \infty} h_{\mathcal{D}}(\bxi_k) = 0$, for which
\begin{equation}\label{20250227_14:41}
\E(F, \bxi_k) \geq h_{\mathcal{D}}(\bxi_k)^{\gamma}.
\end{equation}
\end{corollary}
The sharpness of Corollary \ref{Cor6} is discussed in Section \ref{Qual_Sharp}. The main result of \cite{DNS} presents the upper bound 
\begin{align*}
\E(F, \bxi)   \leq c \left( {\rm Lip}(F) + \sum_{\n \in \Z^N }\left| \widehat{F_{\0}}(\n)\right|   \|\n\|_1 \right)  h_{\mathcal{D}}(\bxi)^{1/2},
\end{align*}
where ${\rm Lip}(F)$ is the Lipschitz constant of $F$ in $\T^N \times \R^N$, and $c>0$ is a constant. In particular, Corollary \ref{Cor6} when $\gamma =1/2$ yields the same quantitative estimate in terms of $h_{\mathcal{D}}(\bxi)^{1/2}$, under less restrictive regularity assumptions on $F$, namely H\"{o}lder-$(1/2)$ continuity on the log-radial part at $\s = 0$, and a fractional $(1/2)$-derivative on the angular part. As we shall see in the proof, the gain on the angular part is a slightly subtler issue. 

\smallskip

In the same spirit of Theorem \ref{Thm4}, it is possible to quantify Bilu's equidistribution if one only assumes $\widehat{F_{\0}} \in L^1(\Z^N)$ for the angular part. For this we work with the tail function 
$$\nu_{\widehat{F_{\0}}}(y):= \sum_{\|n\|_1 >y} \big|\widehat{F_{\0}}(\n)\big|$$
that verifies $\lim_{y \to \infty} \nu_{\widehat{F_{\0}}}(y) =0$.

\begin{theorem}\label{Thm7}
Let $\omega:[0,\infty) \to [0,\infty)$ be a non-decreasing and concave function such that $\lim_{x \to 0} \omega(x) = 0$, and assume that $F: \T^N \times \R^N\to \C$ is continuous and verifies \eqref{20241028_10:36}. Let $W: (0,\infty) \to (0,\infty)$ be a non-decreasing function such that $\lim_{x\to \infty} W(x) = \infty$ and $\lim_{x\to \infty} W(x)/x = 0$. Then, for any $\bxi \in \big(\overline{\mathbb{Q}}^{\times}\big)^N$ we have 
\begin{align*}
\E(F, \bxi)   \leq  \ \omega\big(2 h (\bxi)\big)  + 2\sqrt{6} \sqrt{h_{\mc{D}}(\bxi) \, W\big(h_{\mc{D}}(\bxi)^{-1}\big)}\, \big\|\widehat{F_{\0}}\big\|_{L^1(\Z^N)} + \nu_{\widehat{F_{\0}}} \Big(W\big(h_{\mc{D}}(\bxi)^{-1}\big)\Big).
\end{align*}
\end{theorem}
As an application, with $W(x) = \sqrt{x}$ in Theorem \ref{Thm7} we get 
\begin{align*}
\E(F, \bxi)  \leq   \omega\big(2 h (\bxi)\big) + 2 \sqrt{6}\, h_{\mc{D}}(\bxi)^{1/4}\, \big\|\widehat{F_{\0}}\big\|_{L^1(\Z^N)} + \nu_{\widehat{F_{\0}}} \big(h_{\mc{D}}(\bxi)^{-1/2}\big).
\end{align*}

\smallskip

\subsection{Appendices} At the end of the paper we include three Appendices that complement the material discussed here. Appendix \ref{App_A} presents a bound for the multidimensional angular discrepancy in terms of the generalized height $h_{\mc{D}}$. This is morally in the context of Bilu's equidistribution, but not exactly there since it considers test functions which are characteristic functions of angular sectors (which are not continuous). However, we show that our methods still apply to treat this situation and provide a meaningful upper bound, given in Theorem \ref{Ang_Disc}. 
Appendix \ref{App_B} collects some auxiliary lemmas from the work of D'Andrea, Narv\'aez-Clauss and Sombra \cite{DNS} for the convenience of the reader. Appendix \ref{App_C} briefly completes the proof of Bilu's equidistribution theorem using our effective estimates for a dense subclass of test functions.

\subsection{A word on notation} Throughout the text, $N$-dimensional vectors are denoted with bold font (e.g., $\bxi$, $\n$, $\s$) and numbers with regular font (e.g., $\theta,n,z$). For $\t = (t_1, \ldots, t_N) \in \R^N$ and $1 \leq p <\infty$ we set $\|t\|_p:= \big(\sum_{j=1}^N |t_j|^p\big)^{1/p}$ and $\|\t\|_{\infty} := \max_{1 \leq j \leq N} |t_j|$. We also set $|\t| = \|\t\|_2$ as the usual Euclidean norm. We write $A \lesssim B$ if $A \leq C B$ for a certain constant $C>0 $, and we write $A \simeq B$ if $A \lesssim B$ and $B \lesssim A$ (parameters of dependence of such a constant  $C>0$ appear as a subscript in the inequality sign). For $x \in \R$, $\lfloor x\rfloor$ denotes the integer part of $x$, i.e., the largest integer that is smaller than or equal to $x$.

\section{Consequences of Siegel's lemma}

The following result is a variant of a similar lemma appearing in the work of Petsche \cite[Lemma 2.1]{P} (there, it is stated for a parameter $0<\lambda \leq 1$, while here we restrict to $0<\lambda \leq 1/2$ in order to get a slightly stronger estimate), and is a consequence of Siegel's lemma in the formulation of Bombieri and Vaaler \cite{BV, BV2}. We provide a brief proof for the convenience of the reader.

\begin{lemma}\label{Siegel}
Let $\xi \in \overline{\mathbb{Q}}^{\times}$ be an algebraic number of degree $d = [\Q(\xi):\Q]$, let $0<\lambda \leq 1/2$ be a parameter, and let $m$ be a positive integer large enough so that $\lfloor \lambda m d \rfloor \geq 1$. Then there exists a non-zero polynomial $Q(x)= \sum_{n=0}^{D}b_nx^n \in \mathbb{Z}[x]$ of degree $D$ and with $Q(0)\neq 0$, that satisfies the following conditions: 
  \begin{itemize}
  \item[(i)] $D\leq (1+\lambda)md$.
  \item[(ii)] $Q(x)$ vanishes at $\xi$ with multiplicity at least $m$.
  \item[(iii)] $\log\, \max_{n}|b_n| \leq \frac{3m^2d^2}{2\lfloor \lambda m d\rfloor }\left(h(\xi) + \frac{\log (2d)}{3d}\right)$.  
  \end{itemize}
  \end{lemma}
\begin{proof} Let $L =\lfloor \lambda m d \rfloor$. This statement can be deduced from more general results in the paper of Bombieri and Vaaler \cite{BV2}. In particular, \cite[Corollary 2, Eq. (2.7)]{BV2} establishes the existence of a polynomial $Q(x) \in \Z[x]$, of degree $D \leq N = md + L$, vanishing at $\xi$ with multiplicity at least $m$, and such that 
\begin{align}\label{20241010_10:41}
\log\, \max_{n}|b_n| \leq \frac{Nmd}{L}  h(\xi) + \frac{N^2d}{L} u(m/N)\,,
\end{align}
where $u:(0,1) \to \R$ is a certain function that verifies the bound \cite[Eq. (2.4)]{BV2}
\begin{equation}\label{20241010_10:40}
u(s) \leq \frac{s^2}{2} \big( \log (1/4s) + 3/2\big).
\end{equation}
Parts (i) and (ii) follow directly, while part (iii) follows by applying the bound \eqref{20241010_10:40} into \eqref{20241010_10:41}, together with the fact that $\lambda \leq 1/2$. Finally, one can ensure that $Q(0) \neq 0$ by factoring out the appropriate power of $x$.
\end{proof}

The following lemma plays an important role in our argument. It is an essential tool in refining the regularity results of Petsche \cite{P} and D'Andrea, Narv\'aez-Clauss and Sombra \cite{DNS}. Our proof below brings in some ideas from Soundararajan's proof of the Erd\"{o}s-Tur\'{a}n inequality \cite[Lemmas 1 and 2]{S} in order to refine Petsche's \cite{P} original application of the Bombieri-Vaaler version of Siegel's lemma in this context. A result of similar flavour was recently obtained, via different methods, by Baker and Masser in \cite[Corollary 1.3]{BM}.

 \begin{lemma}\label{New Lemma 1} 
 Let $\xi$ be an algebraic number of degree $d = [\Q(\xi):\Q]$ and let $S = \{\xi = \xi_1, \xi_2, \ldots, \xi_d\} $ be its Galois orbit. Write $\xi_j = |\xi_j|\,e^{2\pi i \theta_j}$. Then 
\begin{align}\label{20241029_19:26}
 \left|\frac{1}{d}\sum_{j=1}^d e^{2\pi i \theta_j}\right| \leq 2 \sqrt{6} \left(h(\xi) + \frac{\log (2d)}{3d}\right)^{\frac{1}{2}}.
 \end{align}
 \end{lemma}
 
 \begin{proof} If the quantity on the right-hand side of \eqref{20241029_19:26} is greater than or equal to $1$, the estimate is trivially true. Assume this is not the case and set
 \begin{align}\label{20241024_15:58}
 \lambda := \sqrt{6} \left(h(\xi) + \frac{\log (2d)}{3d}\right)^{\frac{1}{2}} \leq \frac{1}{2}.
 \end{align}
Applying Lemma \ref{Siegel} we get a polynomial
$Q(x)= \sum_{n=0}^{D}b_nx^n \in \mathbb{Z}[x]$ of degree $D$, with $b_0 \neq 0$, such that $\xi$ is a zero of $Q$ with multiplicity at least $m$. Hence, all conjugates of $\xi$ are also zeros of $Q$ with multiplicity at least $m$. Let $\beta_j= r_j e^{2\pi i \phi_j}$, $j=1, 2, \ldots, D$, be all the zeros of $Q$, where $0\leq \phi_j<1$ for each $j$ (note also that $r_j >0$). Arrange the $\beta_j$'s so that each conjugate of $\xi$ occurs with multiplicity $m$ among the first $md$ roots $\beta_1, \ldots, \beta_{md}$. Now, by the triangle inequality, we have
 \begin{align}\label{20241104_12:08}
  \left|\frac{1}{d}\sum_{j=1}^d e^{2\pi i  \theta_j}\right|\leq  \left|\frac{1}{d}\sum_{j=1}^d e^{2\pi i  \theta_j} -  \frac{1}{md}\sum_{j=1}^D e^{2\pi i  \phi_j}\right| +  \left|\frac{1}{md}\sum_{j=1}^D e^{2\pi i  \phi_j}\right|. 
 \end{align}
Our arrangement of the roots of $Q$ yields
 \begin{align}\label{20241024_14:59}
 \left|\frac{1}{d}\sum_{j=1}^d e^{2\pi i  \theta_j} -  \frac{1}{md}\sum_{j=1}^D e^{2\pi i  \phi_j}\right|= \left|\frac{1}{md}\sum_{j=md+1}^D e^{2\pi i  \phi_j}\right|\leq \frac{D-md}{md}\leq \frac{(1+\lambda)md - md}{md}=\lambda.
 \end{align}
 
 \smallskip
 
 Let $Q_1(z) := \prod_{j=1}^{D}(z- e^{2\pi i \phi_j})$. Observe that, for $|z| =1$, 
 \begin{align*}
\big|z - e^{2\pi i \phi_j}\big|^2  \leq \left| \frac{z}{\sqrt{r_j}} - \sqrt{r_j}\,e^{2\pi i \phi_j} \right|^2\,,
\end{align*}
for each $ j=1, 2, \ldots , D$. Multiplying out we get, for $|z| =1$,
\begin{align}\label{Schur}
|Q_1(z)|  \leq  |Q(z)|/\sqrt{|b_0b_D|}.
\end{align}
For the periodic function $\psi(\theta) = \log|2 \sin(\pi \theta)|$, its Fourier coefficients are given by $\widehat{\psi}(0) = 0$ and $\widehat{\psi}(n) = - \tfrac{1}{2|n|}$ for $n  \in \Z \setminus \{0\}$. Hence, 
\begin{align*}
\begin{split}
&\int_{\R/\Z} e^{2 \pi i  \theta}  \log \big| Q_1\big(e^{2\pi i \theta}\big) \big|\,\d\theta = \sum_{j =1}^D \int_{\R/\Z}e^{2 \pi i \theta}  \log \big| e^{2\pi i \theta} - e^{2\pi i \phi_j}\big|\,\d\theta  \\
& = \sum_{j =1}^D e^{2 \pi i  \phi_j} \int_{\R/\Z} e^{2 \pi i  \tau}  \log \big| e^{2\pi i \tau} - 1\big|\,\d\tau  = \sum_{j =1}^D e^{2 \pi i  \phi_j} \int_{\R/\Z} e^{2 \pi i  \tau}  \log |2 \sin(\pi \tau)|\,\d\tau = - \frac{1}{2}\sum_{j =1}^D e^{2 \pi i \phi_j}.
\end{split}
\end{align*}
Therefore,
\begin{align}\label{20241010_11:45}
\left|\sum_{j =1}^D e^{2 \pi i  \phi_j}\right| &\leq 2 \int_{\R/\Z}  \bigg|\log \big| Q_1\big(e^{2\pi i \theta}\big) \big| \bigg|\,\d\theta \nonumber \\
& = 2\bigg(2 \int_{\R/\Z}  \log^+ \big| Q_1\big(e^{2\pi i \theta}\big) \big|\,\d\theta - \int_{\R/\Z} \log \big| Q_1\big(e^{2\pi i \theta}\big) \big| \d\theta\bigg)\\
& = 4\,\int_{\R/\Z}  \log^+ \big| Q_1\big(e^{2\pi i \theta}\big) \big|\,\d\theta,\nonumber 
\end{align}
where we used the identity $|\log |x| \big| = 2 \log^+ |x| - \log |x|$ in the second passage above, and Jensen's formula in the last passage. Using \eqref{Schur}, the fact that our coefficients $b_n \in \Z$ (with $b_0, b_D \neq 0$), and Lemma \ref{Siegel}, we get, for $|z| =1$, 
\begin{align}\label{20241024_14:53}
\begin{split}
\log^+ |Q_1(z)| & \leq \log^+\big(|Q(z)|/\sqrt{|b_0b_D|}\big)   \leq \log^+|Q(z)| \leq \log\big( (D+1)  \max_{n}|b_n|\big)\\
& \leq \log\big((1+\lambda)md +1\big) + \frac{3m^2d^2}{2\lfloor \lambda m d\rfloor }\left(h(\xi) + \frac{\log (2d)}{3d}\right).
\end{split}
\end{align}
From \eqref{20241010_11:45} and \eqref{20241024_14:53} we then arrive at
 \begin{align}\label{20241024_15:01}
 \left|\frac{1}{md}\sum_{j=1}^D e^{2\pi i  \phi_j}\right|\leq 4 \left( \frac{1}{md}\, \log {\big((1+ \lambda)md +1\big)} +\frac{3md}{2\lfloor\lambda md\rfloor} \bigg(h(\xi) + \frac{\log\, (2d)}{3d}\bigg) \right).
 \end{align}
From \eqref{20241104_12:08}, \eqref{20241024_14:59}, and \eqref{20241024_15:01}, we get
 \begin{align}\label{20241024_15:54}
  \left|\frac{1}{d}\sum_{j=1}^d e^{2\pi i  \theta_j} \right| \leq \lambda + 4 \left( \frac{1}{md}\, \log {\big((1+ \lambda)md +1\big)} +\frac{3md}{2\lfloor\lambda md\rfloor} \bigg(h(\xi) + \frac{\log\, (2d)}{3d}\bigg) \right).
 \end{align}
 Letting $m\rightarrow \infty$ in \eqref{20241024_15:54} yields
  \begin{align*}
  \left|\frac{1}{d}\sum_{j=1}^d e^{2\pi i  \theta_j} \right| \leq \lambda + \frac{6}{\lambda} \bigg(h(\xi) + \frac{\log\, (2d)}{3d}\bigg).
 \end{align*}
 Recalling the choice \eqref{20241024_15:58} for $\lambda$, we arrive at the desired bound. 
\end{proof}

\noindent {\sc Remark}: In the setup of Lemma \ref{New Lemma 1}, we observe that the following bound holds for all $n \in \Z \setminus \{0\}$: 
\begin{align*}
 \left|\frac{1}{d}\sum_{j=1}^d e^{2\pi i n \theta_j}\right| \leq 2 \sqrt{6|n|} \left(h(\xi) + \frac{\log (2d)}{3d}\right)^{\frac{1}{2}}.
 \end{align*}
In fact, letting $S_n$ be the Galois orbit of $\xi_1^n$, and $d_n = |S_n|$, we note that: (i) $d_n \geq d/n$ and (ii) $d = c_n\,d_n$ for some integer $c_n$. Moreover, every $\zeta \in S_n$ is repeated $c_n$ times in the set $\{\xi_j^n, j = 1,2, \ldots, d\}$. Then, applying \eqref{20241029_19:26} to $\zeta_1 = \xi_1^n$ we get 
\begin{align*}
 \left|\frac{1}{d}\sum_{j=1}^d e^{2\pi i n \theta_j}\right|  =  \left|\frac{1}{d_n}\sum_{\zeta \in S_n} \frac{\zeta}{|\zeta|}\right| \leq 2 \sqrt{6} \left(h(\xi_1^n) + \frac{\log (2d_n)}{3d_n}\right)^{\frac{1}{2}} \leq 2 \sqrt{6|n|} \left(h(\xi) + \frac{\log (2d)}{3d}\right)^{\frac{1}{2}},
 \end{align*}
as desired. 

\section{Proofs of Theorems \ref{Thm1} and \ref{Thm4}}

\subsection{Proof of Theorem \ref{Thm1}} We may assume that the constants in \eqref{20250116_14:25} and \eqref{20250116_14:26} are finite, otherwise the estimate is trivially true. 

\subsubsection{Setup} Recall that $S$ is the Galois orbit of $\bxi$. Applying the Fourier inversion formula \eqref{20241025_10:59} and Fubini's theorem, 
\begin{align*}
\E(F, \bxi) &  = \left| \sum_{\n \in \Z^N} \int_{\R^N} \widehat{F}(\n, \t)\,\frac{1}{|S|} \sum_{\balpha \in S} e^{2 \pi i \n \cdot \btheta(\balpha)} \, e^{2 \pi i \t \cdot \s(\balpha)}\,\d \t - \int_{\R^N} \widehat{F}(\0, \t)\,\d \t \right| \\ 
&  = |I_1 + I_2|\\
& \leq |I_1| + |I_2|,
\end{align*}
where
\begin{align*}
I_1:=  \sum_{\n \in \Z^N} \int_{\R^N} \widehat{F}(\n, \t)\,\frac{1}{|S|} \sum_{\balpha \in S} e^{2 \pi i \n \cdot \btheta(\balpha)}  \left(e^{2 \pi i \t \cdot \s(\balpha)}-1\right)\d \t \,,  
\end{align*}
and
\begin{align*}
I_2 := \sum_{\n \in \Z^N\setminus \{\0\}} \int_{\R^N} \widehat{F}(\n, \t)\,\frac{1}{|S|} \sum_{\balpha \in S} e^{2 \pi i \n \cdot \btheta(\balpha)} \d \t  = \sum_{\n \in \Z^N\setminus \{\0\}} \widehat{F_{\0}}(\n)\frac{1}{|S|} \sum_{\balpha \in S} e^{2 \pi i \n \cdot \btheta(\balpha)}.
\end{align*}
Above, for each $\balpha = (\alpha_1, \alpha_2, \ldots, \alpha_N) \in S$, we denote $\btheta(\balpha):= \big(\frac{{\rm arg}(\alpha_1)}{2\pi},  \frac{{\rm arg}(\alpha_2)}{2\pi}, \ldots, \frac{{\rm arg}(\alpha_N)}{2\pi}\big) \in \T^N$ and  $\s(\balpha) = (\log|\alpha_1|, \log|\alpha_2|, \ldots, \log|\alpha_N|) \in \R^N$. We now proceed by bounding $I_1$ and $I_2$ separately.

\subsubsection{Estimate for $I_1$} For generic $\t \in \R^N \setminus \{\0\}$ and $\n \in \Z^N$ we first consider upper bounds for 
\begin{align*}
\left|\frac{1}{|S|} \sum_{\balpha \in S} e^{2 \pi i \n \cdot \btheta(\balpha)}  \left(e^{2 \pi i \t \cdot \s(\balpha)}-1\right)\right|.
\end{align*}
This quantity is trivially bounded above by $2$. For $\delta >0$, let 
\begin{align}\label{20241025_12:09}
\Gamma_{\delta} = \left\{ \z = (z_1, \ldots, z_N) \in (\C^{\times})^N \ : \ \sum_{j=1}^N |\log|z_j|| \leq \delta\right\}.
\end{align}
 We write 
 \begin{align}\label{20241025_12:39}
  \left|\frac{1}{|S|} \sum_{\balpha \in S} e^{2 \pi i \n \cdot \btheta(\balpha)}  \left(e^{2 \pi i \t \cdot \s(\balpha)}-1\right)\right| \leq   \left|\frac{1}{|S|} \sum_{\balpha \in S \cap \Gamma_{\delta}} e^{2 \pi i \n \cdot \btheta(\balpha)}  \left(e^{2 \pi i \t \cdot \s(\balpha)}-1\right)\right| + \frac{2}{|S|}\sum_{\balpha \in S \setminus \Gamma_{\delta}} 1.
 \end{align}
From \eqref{20241025_12:09} and Lemma \ref{Lem11_App} in Appendix \ref{App_B},
 \begin{align}\label{20241025_12:40}
 \frac{2}{|S|}\sum_{\balpha \in S \setminus \Gamma_{\delta}} 1 \leq  \frac{2}{\delta} \left(\frac{1}{|S|}\sum_{\balpha \in S} \|\s(\balpha)\|_1\right)\leq \frac{4  h(\bxi)}{\delta}. 
 \end{align}
 On the other hand, note that 
 \begin{align}\label{20241025_12:41}
 \begin{split}
&  \left|\frac{1}{|S|} \sum_{\balpha \in S \cap \Gamma_{\delta}} e^{2 \pi i \n \cdot \btheta(\balpha)}  \left(e^{2 \pi i \t \cdot \s(\balpha)}-1\right)\right| \leq \frac{1}{|S|} \sum_{\balpha \in S \cap \Gamma_{\delta}}  \left|e^{2 \pi i \t \cdot \s(\balpha)}-1\right|\\
 &\qquad  \leq \frac{1}{|S|} \sum_{\balpha \in S \cap \Gamma_{\delta}} |2 \pi  \t \cdot \s(\balpha)| \leq \frac{2\pi \|\t\|_{\infty}}{|S|} \sum_{\balpha \in S \cap \Gamma_{\delta}} \|\s(\balpha)\|_1 \leq 2\pi \|\t\|_{\infty} \delta. 
 \end{split}
 \end{align}
 Combining \eqref{20241025_12:39}, \eqref{20241025_12:40} and \eqref{20241025_12:41} we get
 \begin{align}\label{20241025_12:50}
  \left|\frac{1}{|S|} \sum_{\balpha \in S} e^{2 \pi i \n \cdot \btheta(\balpha)}  \left(e^{2 \pi i \t \cdot \s(\balpha)}-1\right)\right| \leq  \frac{4  h(\bxi)}{\delta} + 2\pi \|\t\|_{\infty} \delta.
 \end{align}
 The optimal choice of $\delta = \big(2h(\bxi) /(\pi \|\t\|_{\infty})\big)^{1/2}$ in \eqref{20241025_12:50}, together with the trivial bound of $2$, yields 
 \begin{align}\label{20241025_13:33}
 \left|\frac{1}{|S|} \sum_{\balpha \in S} e^{2 \pi i \n \cdot \btheta(\balpha)}  \left(e^{2 \pi i \t \cdot \s(\balpha)}-1\right)\right| \leq \min\big\{2\sqrt{8 \pi  h(\bxi)\|\t\|_{\infty}}\,,\, 2\big\}.
 \end{align}

Using \eqref{20241025_13:33}, we now bound $I_1$ as follows:
 \begin{align}\label{20241025_13:51}
 \begin{split}
 |I_1| & \leq \sum_{\n \in \Z^N} \int_{\R^N} \left|\widehat{F}(\n, \t)\right| \left| \frac{1}{|S|} \sum_{\balpha \in S} e^{2 \pi i \n \cdot \btheta(\balpha)}  \left(e^{2 \pi i \t \cdot \s(\balpha)}-1\right)\right|\d \t\\
 & \leq 2\sqrt{8 \pi  h(\bxi)}\sum_{\n \in \Z^N} \int_{\|\t\|_{\infty} \leq M} \left| \widehat{F}(\n, \t)\right| \sqrt{\|\t\|_{\infty}}\, \d\t + 2 \sum_{\n \in \Z^N} \int_{\|\t\|_{\infty} > M} \left| \widehat{F}(\n, \t)\right| \d\t\,,
 \end{split}
 \end{align}
 where $M$ is a parameter that will be chosen later. Since $G(x)/\sqrt{x}$ is non-increasing, we have, for $\|\t\|_{\infty} \leq M$,
 \begin{align}\label{20241025_13:49}
  \sqrt{\|\t\|_{\infty}}\leq G( \|\t\|_{\infty}) \frac{\sqrt{M}}{G(M)}. 
 \end{align}
 Also, since $G$ is non-decreasing, we have, for $\|\t\|_{\infty} > M$,
  \begin{align}\label{20241025_13:50}
  1 \leq  \frac{G(\|\t\|_{\infty})}{G(M)}. 
 \end{align}
 Plugging \eqref{20241025_13:49} and \eqref{20241025_13:50} into \eqref{20241025_13:51}, we get
 \begin{align*}
 |I_1| \leq 2\sqrt{8 \pi  h(\bxi)}  \frac{\sqrt{M}}{G(M)} & \sum_{\n \in \Z^N} \int_{\|\t\|_{\infty} \leq M} \left| \widehat{F}(\n, \t)\right| G( \|\t\|_{\infty})\, \d\t \\
 & + \frac{2}{G(M)} \sum_{\n \in \Z^N} \int_{\|\t\|_{\infty} > M} \left| \widehat{F}(\n, \t)\right|G( \|\t\|_{\infty})\, \d\t.
 \end{align*}
 With the choice of $M = (8 \pi  h(\bxi))^{-1}$ we arrive at 
 \begin{align*}
  |I_1| \leq \frac{2}{G\big( (8 \pi  h(\bxi))^{-1}\big)} \sum_{\n \in \Z^N} \int_{\R^N} \left| \widehat{F}(\n, \t)\right|G( \|\t\|_{\infty})\, \d\t.
 \end{align*}
 
 \subsubsection{Estimate for $I_2$} \label{Est_I_2} For $\n \neq \0$, let $S_{\n} \subset \C^{\times}$ be the Galois orbit of $\chi^{\n}(\bxi) = \xi_1^{n_1}\xi_2^{n_2}\ldots \xi_N^{n_N}$. Then, by Lemma \ref{Lem10} in Appendix \ref{App_B}, there exists an integer $c_{\n}$ such that $|S| = c_{\n} \cdot |S_{\n}|$. Furthermore, every $\zeta \in S_{\n}$ is repeated $c_{\n}$ in the set of algebraic numbers $\{\chi^{\n}(\balpha)\ ; \ \balpha \in S\}$. Therefore, for $\n \neq \0$, we may apply Lemma \ref{New Lemma 1}  to get
 \begin{align}\label{20241025_16:27}
\left|\frac{1}{|S|} \sum_{\balpha \in S} e^{2 \pi i \n \cdot \btheta(\balpha)} \right|=  \left|\frac{1}{|S|} \sum_{\balpha \in S}  \frac{\chi^{\n}(\balpha)}{\big|\chi^{\n}(\balpha)\big|} \right|= \left| \frac{1}{|S_{\n}|} \sum_{\zeta \in S_{\n}}  \frac{\zeta}{|\zeta|}\right| \leq 2\sqrt{6}\left(h\big(\chi^{\n}(\bxi)\big) + \frac{\log (2|S_{\n}|)}{3|S_{\n}|}\right)^{\frac{1}{2}}.
 \end{align}
 Note that 
 \begin{align}\label{20241025_16:28}
 h\big(\chi^{\n}(\bxi)\big) = h(\xi_1^{n_1}\ldots \xi_N^{n_N}) \leq h(\xi_1^{n_1}) + \ldots + h(\xi_N^{n_N}) = |n_1|h(\xi_1) + \ldots + |n_N|h(\xi_N) \leq \|\n\|_{\infty} h(\bxi).  
 \end{align}
 Also, from the definition of $\mc{D}(\bxi)$ we have $|S_{\n}| \geq \mc{D}(\bxi) / \|\n\|_1$. Since $x \mapsto  \log(2x) / (3x)$ is decreasing for $x\geq 2$, one can check that
 \begin{align}\label{20241025_16:22}
 \frac{\log (2|S_{\n}|)}{3|S_{\n}|}  \leq  \|\n\|_1\frac{\log \big(2 \mc{D}(\bxi)\big)}{3 \mc{D}(\bxi)}
 \end{align}
(it is convenient to consider two cases: $1 \leq \|\n\|_1\leq  \mc{D}(\bxi)/2$ and $\|\n\|_1>  \mc{D}(\bxi)/2$ when showing \eqref{20241025_16:22}). From \eqref{20241025_16:27}, \eqref{20241025_16:28} and \eqref{20241025_16:22}, noting that $\|\n\|_{\infty} \leq \|\n\|_1$, we arrive at
\begin{align*}
\left|\frac{1}{|S|} \sum_{\balpha \in S} e^{2 \pi i \n \cdot \btheta(\balpha)} \right| \leq2 \sqrt{6\|\n\|_1}    \left(h(\bxi) + \frac{\log\big(2\mathcal{D}(\bxi)\big)}{3 \mathcal{D}(\bxi)}\right)^{1/2}.
\end{align*}
Noting the trivial bound of $1$, and the definition of $h_{\mathcal{D}} (\bxi)$ in \eqref{20241104_15:14}, we have just showed that 
 \begin{align}\label{20241025_16:37}
\left|\frac{1}{|S|} \sum_{\balpha \in S} e^{2 \pi i \n \cdot \btheta(\balpha)} \right| \leq  \min\big\{ 2\sqrt{6  \,h_{\mathcal{D}} (\bxi) \|\n\|_1} \, , \, 1\big\}.
\end{align}
 
Using \eqref{20241025_16:37} we bound $I_2$ as follows:
\begin{align}\label{20241025_16:41}
\begin{split}
|I_2| & \leq \sum_{\n \in \Z^N\setminus \{\0\}} \left| \widehat{F_{\0}}(\n)\right| \left|\frac{1}{|S|} \sum_{\balpha \in S} e^{2 \pi i \n \cdot \btheta(\balpha)}\right| \\
& \leq 2\sqrt{6 \, h_{\mathcal{D}} (\bxi)}  \sum_{0 < \|\n\|_1 \leq M} \left| \widehat{F_{\0}}(\n)\right|\sqrt{\|\n\|_1} + \sum_{\|\n\|_1 > M} \left| \widehat{F_{\0}}(\n)\right|\,,
\end{split}
\end{align}
where $M$ is a parameter that will be chosen later. Since $H(x)/\sqrt{x}$ is non-increasing, we have, for $\|\n\|_{1} \leq M$,
 \begin{align}\label{20241025_16:42}
 \sqrt{\|\n\|_1}  \leq H( \|\n\|_{1}) \frac{\sqrt{M}}{H(M)}.
 \end{align}
  Also, since $H$ is non-decreasing, we have, for $\|\n\|_{1} > M$,
  \begin{align}\label{20241025_16:43}
  1 \leq  \frac{H(\|\n\|_{1})}{H(M)}. 
 \end{align}
  Plugging \eqref{20241025_16:42} and \eqref{20241025_16:43} into \eqref{20241025_16:41}, we get
 \begin{align*}
|I_2| \leq 2\sqrt{6 \, h_{\mathcal{D}} (\bxi)}  \frac{\sqrt{M}}{H(M)} \sum_{0 < \|\n\|_1 \leq M} \left| \widehat{F_{\0}}(\n)\right|H( \|\n\|_{1}) + \frac{1}{H(M)} \sum_{\|\n\|_1 > M} \left| \widehat{F_{\0}}(\n)\right|H( \|\n\|_{1}). 
 \end{align*}
 With the choice of $M = (24 \,h_{\mathcal{D}} (\bxi))^{-1}$ we arrive at 
 \begin{align*}
 |I_2| \leq \frac{1}{H \big( (24 \,h_{\mathcal{D}} (\bxi))^{-1} \big)}\sum_{\n \in \Z^N\setminus \{\0\}} \left| \widehat{F_{\0}}(\n)\right|H( \|\n\|_{1}).
 \end{align*}
 This concludes the proof of Theorem \ref{Thm1}.

 \subsection{Proof of Theorem \ref{Thm4}} This requires only a few modifications in comparison to the proof of Theorem \ref{Thm1}. The setup is the same and the estimate for $I_1$ is the same up to equation \eqref{20241025_13:51}. In that equation we now choose $M = W\big(h_{\mc{D}}(\bxi)^{-1}\big)$. This leads us to 
  \begin{align} \label{20241026_18:22}
 \begin{split}
 |I_1| & \leq 2\sqrt{8 \pi  h(\bxi)}\sum_{\n \in \Z^N} \int_{\|\t\|_{\infty} \leq M} \left| \widehat{F}(\n, \t)\right| \sqrt{\|\t\|_{\infty}}\, \d\t + 2 \sum_{\n \in \Z^N} \int_{\|\t\|_{\infty} > M} \left| \widehat{F}(\n, \t)\right| \d\t \\
 & \leq 2\sqrt{8 \pi} \sqrt{h(\bxi) \, W\big(h_{\mc{D}}(\bxi)^{-1}\big)}\,  \big\|\widehat{F}\big\|_{L^1(\Z^N \times \R^N)} + 2 \nu_{\widehat{F}} \Big(W\big(h_{\mc{D}}(\bxi)^{-1}\big)\Big)\\
 & \leq 2\sqrt{8 \pi} \sqrt{h_{\mc{D}}(\bxi) \, W\big(h_{\mc{D}}(\bxi)^{-1}\big)}\,  \big\|\widehat{F}\big\|_{L^1(\Z^N \times \R^N)} + 2 \nu_{\widehat{F}} \Big(W\big(h_{\mc{D}}(\bxi)^{-1}\big)\Big).
 \end{split}
 \end{align}
 Similarly, the estimate for $I_2$ is the same up to \eqref{20241025_16:41}, and we now choose $M = W\big(h_{\mc{D}}(\bxi)^{-1}\big)$ in that equation to get
 \begin{align}\label{20241026_18:23}
\begin{split}
|I_2| &  \leq 2\sqrt{6 \, h_{\mathcal{D}} (\bxi)}  \sum_{0 < \|\n\|_1 \leq M} \left| \widehat{F_{\0}}(\n)\right|\sqrt{\|\n\|_1} + \sum_{\|\n\|_1 > M} \left| \widehat{F_{\0}}(\n)\right| \\
& \leq 2\sqrt{6} \sqrt{h_{\mc{D}}(\bxi) \, W\big(h_{\mc{D}}(\bxi)^{-1}\big)}\, \big\|\widehat{F}\big\|_{L^1(\Z^N \times \R^N)} + \nu_{\widehat{F}}  \Big(W\big(h_{\mc{D}}(\bxi)^{-1}\big)\Big).
\end{split}
\end{align}
 Combining \eqref{20241026_18:22} and \eqref{20241026_18:23} we arrive at the proposed result.
 
 \section{Proofs of Theorems \ref{Thm5} and \ref{Thm7}} 
 
  \subsection{Proof of Theorem \ref{Thm5}} We may assume that the constant in \eqref{20250116_14:27} is finite, otherwise the estimate is trivially true. In this case, observe that $\widehat{F_{\0}} \in L^1(\Z^N)$ and Fourier inversion for $F_{\0}$ holds.
 
 \subsubsection{Setup} Recall that $S$ is the Galois orbit of $\bxi$. We add and subtract a suitable term to get
\begin{align*}
\E(F, \bxi)  &= \left| \frac{1}{|S|} \sum_{\balpha \in S} F\big(\btheta(\balpha), \s(\balpha)\big) -  \frac{1}{|S|} \sum_{\balpha \in S} F\big(\btheta(\balpha), \0\big) + \frac{1}{|S|} \sum_{\balpha \in S} F(\btheta(\balpha), \0) - \int_{\T^N} F\big(\btheta, \0\big)  \,\d \btheta \right| \\
&= | I_1 + I_2|\\
& \leq |I_1| + |I_2| \,,
\end{align*}
where
\begin{align*}
I_1 =  \frac{1}{|S|} \sum_{\balpha \in S} F\big(\btheta(\balpha), \s(\balpha)\big) -  \frac{1}{|S|} \sum_{\balpha \in S} F\big(\btheta(\balpha), \0\big)\,,
\end{align*}
and
\begin{align*}
I_2 = \frac{1}{|S|} \sum_{\balpha \in S} F(\btheta(\balpha), \0) - \int_{\T^N} F\big(\btheta, \0\big)  \,\d \btheta = \sum_{\n \in \Z^N\setminus \{\0\}} \widehat{F_{\0}}(\n)\frac{1}{|S|} \sum_{\balpha \in S} e^{2 \pi i \n \cdot \btheta(\balpha)}.
\end{align*}
Note the use of Fourier inversion for $F_{\0}(\btheta)= F(\btheta, \0)$  in the last passage above.

\subsubsection{Estimate for $I_1$}  Using the triangle inequality and \eqref{20241028_10:36}, the fact that $\omega$ is non-decreasing and concave (for Jensen's inequality below), and Lemma \ref{Lem11_App} in Appendix \ref{App_B}, we get 
\begin{align*}
|I_1|  \leq \frac{1}{|S|} \sum_{\balpha \in S} \omega\big( | \s(\balpha)|\big) \leq \frac{1}{|S|} \sum_{\balpha \in S} \omega\big( \| \s(\balpha)\|_1\big) \leq \omega\left( \frac{1}{|S|} \sum_{\balpha \in S}\|\s(\balpha)\|_1\right) \leq \omega\big(2 h (\bxi)\big). 
\end{align*}

\subsubsection{Estimate for $I_2$} Note that the quantity $I_2$ is the same as in the proof of Theorem \ref{Thm1}, and the desired estimate is done in \S \ref{Est_I_2}.

  \subsection{Proof of Theorem \ref{Thm7}} The setup and the estimate for $I_1$ are the same as in the proof of Theorem \ref{Thm5}. The estimate for $I_2$ is the same as \eqref{20241026_18:23}.
  
 \section{Sharpness of Corollaries \ref{Cor3} and \ref{Cor6}}\label{Qual_Sharp}
  
In this section we show that the estimates in Corollaries \ref{Cor3} and \ref{Cor6} are quantitatively sharp by presenting explicit examples.  

\subsection{Sharpness of Corollary \ref{Cor3}} \label{Sec_5.1} Consider a function $F: \T^N \times \R^N \to \C$ such that $\widehat{F}(\n,\t) = 0$ for $\n \neq \0$ and $\t \in \R^N$ and 
$$\widehat{F}(\0,\t) = \frac{1}{(1 + |\t|^2)^{(N + \gamma)/2} \ \log (20 + |\t|^2)\ (\log \log (20 + |\t|^2))^2}$$
for $\t \in \R^N$. Note that $\widehat{F} \in L^1(\Z^N \times \R^N)$ and \eqref{20241026_15:38} is finite for such $F$. Note also that all partial derivatives of $\t \mapsto \widehat{F}(\0,\t)$ belong to $L^1(\R^N)$ and vanish at infinity. Letting $\s = (s_1, \ldots, s_N)$ as usual, from integration by parts we get 
$$F(\btheta, \s) := \int_{\R^N} \widehat{F}(\0,\t)\,e^{2\pi i \t \cdot \s} \d \t = \frac{1}{(2\pi i s_j)^{N+1}} \int_{\R^N} \left(\frac{\partial^{N+1}}{\partial t_j^{N+1}}\widehat{F}(\0,\t)\right)\,e^{2\pi i \t \cdot \s} \d \t\,,$$
whenever $s_j \neq 0$. This plainly implies that $|\s|^{N+1} |F(\btheta, \s)|$ is bounded, and hence $F \in L^p( \T^N \times \R^N)$ for all $1 \leq p \leq \infty$. Note that $F$ is continuous by the Riemann-Lebesgue lemma, and therefore $F \in \mc{A}$. 

\smallskip

Let $k_0 = k_0(N)$ be a large natural number. For $k \geq k_0$, consider the polynomials $P_{k,j}(x) = x^{d_{k,j}} - d_{k,j}$ for $j=1,2,\ldots, N$, where $d_{k,j}$'s are prime numbers with $2^k < d_{k,1} < d_{k,2} < \ldots < d_{k,N} < 2^{k+1}$. Note that such polynomials are irreducible by Eisenstein's criterion. Let $\xi_{k, j}$ be a root of $P_{k,j}$, and set $\bxi_k = (\xi_{k,1}, \xi_{k,2}, \ldots, \xi_{k,N})$. Let $M_{k,j} := d_{k,j} / (\log d_{k,j})$ and $M_k := 2^k/ (\log 2^k)$. Note that we have $M_{k,j} \simeq M_k$ and $h_{\mc{D}}(\bxi) \simeq_N M_k^{-1}$. By Fourier inversion,
\begin{align}\label{20241105_15:07}
\begin{split}
 \E(F, \bxi_k)  &=  \left| \int_{\R^N} \widehat{F}(\0,\t) \left( e^{2\pi i \big(\frac{t_1}{M_{k,1}} + \ldots + \frac{t_N}{M_{k,N}}\big)} - 1\right)\d \t\right|\\
& = 2 \int_{\R^N} \widehat{F}(\0,\t) \sin^2\!\Big( \pi  \big(\tfrac{t_1}{M_{k,1}} + \ldots + \tfrac{t_N}{M_{k,N}}\big) \Big)\d \t  \\
& \geq 2 \int_{(\R^+)^N} \widehat{F}(\0,\t) \sin^2\!\Big( \pi  \big(\tfrac{t_1}{M_{k,1}} + \ldots + \tfrac{t_N}{M_{k,N}}\big) \Big)\d \t.
\end{split}
\end{align}
Recall that $\sin (\pi x/M_k) \simeq \pi x/M_k$ for $x \in [0,M_k/2]$. Hence, 
\begin{align}\label{20241105_15:08}
  \int_{(\R^+)^N} &  \widehat{F}(\0,\t)   \sin^2\!\Big( \pi  \big(\tfrac{t_1}{M_{k,1}} + \ldots + \tfrac{t_N}{M_{k,N}}\big) \Big)\d \t  \nonumber  \geq \int_{\substack{\t \in (\R^+)^N \\ \|\t\|_1 \leq M_k/2}} \widehat{F}(\0,\t) \sin^2\!\Big( \pi  \big(\tfrac{t_1}{M_{k,1}} + \ldots + \tfrac{t_N}{M_{k,N}}\big) \Big)\d \t \nonumber \\
& \gtrsim \int_{\substack{\t \in (\R^+)^N \\ \|\t\|_1 \leq M_k/2}}\frac{\Big( \pi  \big(\tfrac{t_1}{M_{k,1}} + \ldots + \tfrac{t_N}{M_{k,N}}\big) \Big)^2}{(1 + |\t|^2)^{(N + \gamma)/2} \ \log (20 + |\t|^2) \ (\log \log (20 + |\t|^2))^2}\d \t  \nonumber \\
&  \gtrsim  \frac{1 }{M_k^2\, (\log M_k)\, (\log \log M_k)^2 }\int_{\substack{\t \in (\R^+)^N \\ \|\t\|_1 \leq M_k/2}}\frac{\|\t\|_1^2}{(1 + |\t|^2)^{(N + \gamma)/2} } \,\d \t \nonumber \\
& \gtrsim_N  \frac{1 }{M_k^2 \,(\log M_k)\, (\log \log M_k)^2 }\int_{\substack{\t \in (\R^+)^N \\ |\t| \leq M_k/2}}\frac{|\t|^2}{(1 + |\t|^2)^{(N + \gamma)/2} } \,\d \t  \\
& \simeq_N  \frac{1 }{M_k^2 \,(\log M_k)\, (\log \log M_k)^2 }\int_{ |\t| \leq M_k/2}\frac{|\t|^2}{(1 + |\t|^2)^{(N + \gamma)/2} } \,\d \t \nonumber \\
& \gtrsim_N  \frac{1 }{M_k^2 \, (\log M_k)\, (\log \log M_k)^2 } \int_1^{M_k/2} r^{1 - \gamma }\,\d r  \gtrsim \frac{1 }{M_k^{\gamma} \,(\log M_k)\, (\log \log M_k)^2 }. \nonumber
\end{align}
Note the change to polar coordinates in the second to last inequality in \eqref{20241105_15:08}. From \eqref{20241105_15:07} and \eqref{20241105_15:08} we arrive at 
\begin{align*}
 \E(F, \bxi_k) \gtrsim_N \frac{h_{\mathcal{D}}(\bxi_k)^{\gamma}}{|\log h_{\mathcal{D}}(\bxi_k)|\, \big(\log |\log h_{\mathcal{D}}(\bxi_k)|\big)^2 }.
\end{align*}
We may now simply renormalize the function $F$, multiplying it by an appropriate constant depending on $N$, to arrive at the formulation proposed in \eqref{20250227_14:39}.

\subsection{Sharpness of Corollary \ref{Cor6}} Let $F: \T^N \times \R^N \to \C$ be given by $F(\btheta, \s) = |\s|^{\gamma}$. In this case $F_{\0}(\btheta) =F(\btheta, \0)= 0$ for any $\btheta \in \T^N$. With $\bxi_k$ and $M_k$ as in \S \ref{Sec_5.1}, a direct computation yields
\begin{align*}
 \E(F, \bxi_k)\simeq_N M_k^{-\gamma} \simeq_N h_{\mathcal{D}}(\bxi_k)^{\gamma}.
\end{align*} 
Again, we may simply renormalize the function $F$, multiplying it by an appropriate constant depending on $N$, to arrive at the formulation proposed in \eqref{20250227_14:41}.

\appendix
\numberwithin{equation}{section}

\section{Bounds for the multidimensional angular discrepancy} \label{App_A}
Let $\bxi \in \big(\overline{\mathbb{Q}}^{\times}\big)^N$ and let $S$ be its Galois orbit. For each $\balpha = (\alpha_1, \alpha_2, \ldots, \alpha_N) \in S$, we denote $\btheta(\balpha):= \big(\frac{{\rm arg}(\alpha_1)}{2\pi},  \frac{{\rm arg}(\alpha_2)}{2\pi}, \ldots, \frac{{\rm arg}(\alpha_N)}{2\pi}\big) \in \T^N$. We define the multidimensional angular discrepancy of the Galois orbit of $\bxi$ as 
\begin{equation*}
\Delta(\bxi) := \sup_{I \subset \T^N} \left| \frac{1}{|S|} \sum_{\balpha \in S} \chi_{I}\big(\btheta(\balpha)\big) - |I|\right|\,,
\end{equation*}
where the supremum is taken over all subsets $I$ of $\T^N$ that are products of intervals, i.e., $I = \prod_{j=1}^N I_j$ where each $I_j \subset \T$ is an interval, and $\chi_{I}$ denotes the characteristic function of the set $I$. Moreover, $|I|$ denotes the volume of $I$, i.e., $|I|:=  \prod_{j=1}^N |I_j|$, where $|I_j|$ is the usual measure of the interval $I_j\subset \T$.

\smallskip

Since the function $\chi_I$ is not continuous, our main results in the Introduction do not directly apply. Nevertheless, we can still apply the methods developed here to arrive at the following upper bound. 

\begin{theorem}\label{Ang_Disc}
Let $\bxi \in \big(\overline{\mathbb{Q}}^{\times}\big)^N$. Then, for $h_{\mc{D}}(\bxi) \leq e^{-1}$, we have
\begin{equation}\label{20241120_11:36}
\Delta(\bxi) \leq \Big( 9 \left(\tfrac{3}{2}\right)^N + 14N\Big)  \,h_{\mc{D}}(\bxi)^{1/3}\,\big| \log h_{\mc{D}}(\bxi)\big|^{2(N-1)/3}.
\end{equation}
\end{theorem}

This is naturally connected to the classical Erd\"{o}s-Tur\'an inequality \cite{ET} for the angular discrepancy of zeros of polynomials; see also \cite{AM, C_etal_2021, G, M, SW, S}. When $N=1$, Theorem \ref{Ang_Disc} recovers a well-known upper bound that dates back to the works of Langevin \cite{La} and Mignotte \cite{Mi}. This result has been recently sharpened by Baker and Masser \cite[Theorem 1.10]{BM}, who were able to remove the dependence on the degree on the upper bound, i.e., replacing $h_{\mc{D}}^{1/3}$ by simply $h^{1/3}$ on the right-hand side of \eqref{20241120_11:36} when $\xi$ is not a root of unity. In the multidimensional case, it is an interesting problem whether or not one can remove the logarithmic term and the dependence on the generalized degree on the right-hand side of \eqref{20241120_11:36}. A result of similar flavour to Theorem \ref{Ang_Disc} was obtained by D'Andrea, Galligo and Sombra in \cite[Theorem 1.4]{DNS14} in the context of angular discrepancy for solutions of systems of sparse polynomial equations.

\begin{proof}
We make use of the multidimensional Erd\"os-Tur\'an-Koksma inequality \cite[Theorem 1.21]{DT}. It states that, for any non-negative integer $M$, we have
\begin{align}\label{20241120_11:09}
\Delta(\bxi) \leq \bigg(\frac{3}{2}\bigg)^N \left(\frac{2}{M+1} + \sum_{0 < \|{\n}\|_{\infty} \leq M} \frac{1}{r(\n)} \left|\frac{1}{|S|} \sum_{\balpha \in S} e^{2 \pi i \n \cdot \btheta(\balpha)}  \right|  \right),
\end{align}
where $r(\n):= \prod_{i=1}^N\max\{1, |n_i|\}$, for $\n = (n_1, n_2, \ldots, n_N) \in \Z^N$. When $M=0$ we regard the sum on the right-hand side of  \eqref{20241120_11:09} as empty. Using the bound \eqref{20241025_16:37} we get
\begin{align}\label{20241120_12:18}
\Delta(\bxi) \leq \bigg(\frac{3}{2}\bigg)^N \left(\frac{2}{M+1} + 2\sqrt{6\,  h_{\mathcal{D}} (\bxi) } \sum_{0 < \|{\n}\|_{\infty} \leq M} \frac{\sqrt{\|\n\|_1}}{r(\n)}  \right).
\end{align}
Using the bound $\sqrt{\|\n\|_1} \leq \sum_{j=1}^N \sqrt{|n_j|}$ and the symmetries in the sum, we have 
\begin{align}\label{20241120_12:08}
\begin{split}
&\sum_{0 < \|{\n}\|_{\infty} \leq M} \frac{\sqrt{\|\n\|_1}}{r(\n)}  \leq  \sum_{0 < \|{\n}\|_{\infty} \leq M} \frac{\sqrt{|n_1|} + \ldots + \sqrt{|n_N|}}{r(\n)}  = N \sum_{0 < \|{\n}\|_{\infty} \leq M} \frac{\sqrt{|n_1|}}{r(\n)} \\
& = N \sum_{0 \leq \|{\n}\|_{\infty} \leq M} \frac{\sqrt{|n_1|}}{r(\n)} = N \left(\sum_{0 \leq |n_1| \leq M} \frac{\sqrt{|n_1|}}{\max\{1, |n_1|\}}\right)\left(\sum_{0 \leq |k| \leq M} \frac{1}{\max\{1, |k|\}}\right)^{N-1}\\
& \leq N \,\big(4 \sqrt{M}\big) \big(3 + 2\log M\big)^{N-1}\,,
\end{split}
\end{align}
for $M \geq 1$ (if $M=0$ the penultimate line in \eqref{20241120_12:08} is simply zero, and we understand the upper bound in the last line as zero). Combining \eqref{20241120_12:18} and \eqref{20241120_12:08} we arrive at
\begin{align}\label{20241120_13:57}
\Delta(\bxi) \leq 2 \bigg(\frac{3}{2}\bigg)^N \left( \frac{1}{M+1} + \sqrt{c\, h_{\mathcal{D}} (\bxi)} \sqrt{M}\,\big(\tfrac{3}{2} + \log M\big)^{N-1} \right)\,,
\end{align}
where $c = 6 \, N^2 \, 2^{2N+2}$. 

\smallskip

The idea now is to choose 
$$M := \Big\lfloor e^{-3/2}  \,h_{\mathcal{D}} (\bxi)^{-1/3} \,\big| \log h_{\mathcal{D}} (\bxi)\big|^{-2(N-1)/3}\Big \rfloor.$$
One then immediately has
\begin{equation}\label{20241120_13:56}
\frac{1}{M+1} \leq e^{3/2} \,h_{\mathcal{D}} (\bxi)^{1/3} \,\big| \log h_{\mathcal{D}} (\bxi)\big|^{2(N-1)/3}.
\end{equation}
Since $h_{\mc{D}}(\bxi) \leq e^{-1}$, we have $\big| \log h_{\mathcal{D}} (\bxi)\big| \geq 1$ and hence
\begin{align}\label{20241120_13:39}
M \leq  \Big\lfloor e^{-3/2} \, h_{\mathcal{D}} (\bxi)^{-1/3} \Big \rfloor \leq e^{-3/2} \, h_{\mathcal{D}} (\bxi)^{-1/3}.
\end{align}
Therefore, if $M \geq 1$, we may use \eqref{20241120_13:39} to get
\begin{align*}
\big(\tfrac{3}{2} + \log M\big)^{N-1} & \leq \left(\tfrac{3}{2} + \log \big(e^{-3/2} \, h_{\mathcal{D}} (\bxi)^{-1/3}\big)\right)^{N-1} = \left( \tfrac{1}{3} \big| \log h_{\mathcal{D}} (\bxi)\big| \right)^{N-1}\,,
\end{align*}
which yields
\begin{align}\label{20241120_13:55}
\sqrt{c\, h_{\mathcal{D}} (\bxi)} \sqrt{M}\,\big(\tfrac{3}{2} + \log M\big)^{N-1} \leq  \left(\tfrac{1}{3}\right)^{N-1}  e^{-3/4}\, c^{1/2}\, h_{\mathcal{D}} (\bxi)^{1/3} \,\big| \log h_{\mathcal{D}} (\bxi)\big|^{2(N-1)/3}.
\end{align}
Combining \eqref{20241120_13:57}, \eqref{20241120_13:56} and \eqref{20241120_13:55} we arrive at 
\begin{align}\label{20241120_14:03}
\Delta(\bxi) \leq 2 \left(\frac{3}{2}\right)^N \left(e^{3/2} + \left(\tfrac{1}{3}\right)^{N-1}  e^{-3/4}\, c^{1/2}\right) h_{\mc{D}}(\bxi)^{1/3}\,\big| \log h_{\mc{D}}(\bxi)\big|^{2(N-1)/3}
\end{align}
(note that when $M=0$ we do not have the rightmost term in \eqref{20241120_13:57}, and \eqref{20241120_14:03} continues to hold). Recalling that $c = 6 \, N^2 \, 2^{2N+2}$, this simplifies to
\begin{align*}
\Delta(\bxi) & \leq  \left(2 e^{3/2}\left(\tfrac{3}{2}\right)^N +  12\sqrt{6} \,e^{-3/4}\,N\right) h_{\mc{D}}(\bxi)^{1/3}\,\big| \log h_{\mc{D}}(\bxi)\big|^{2(N-1)/3}.
\end{align*}
Noting that $2 e^{3/2} < 9$ and $12\sqrt{6} \,e^{-3/4} < 14$, we arrive at \eqref{20241120_11:36}. 
\end{proof}

\section{Some auxiliary lemmas} \label{App_B}
In this appendix we collect some auxiliary lemmas from the work of D'Andrea, Narv\'aez-Clauss and Sombra \cite{DNS}. For the convenience of the reader, we briefly reproduce the proofs. 

\begin{lemma}[cf. Lemma 2.3 of \cite{DNS}] \label{Lem10}Let $\bxi = (\xi_1, \ldots, \xi_N) \in \big(\overline{\mathbb{Q}}^{\times}\big)^N$ and $S$ be its Galois orbit. Then 
\begin{itemize}
\item[(i)] $|S| = [\Q(\xi_1, \ldots, \xi_N) \, : \, \Q]$.
\item[(ii)] ${\rm deg}(\chi^{\n}(\bxi))$ divides $|S|$ for every $\n \in \Z^N$.
\end{itemize}
\end{lemma}
\begin{proof}
Note that ${\rm Gal}\big(\overline{\Q}/\Q\big)$ acts transitively on $S$ and that ${\rm Gal}\big(\overline{\Q}/\Q(\xi_1, \ldots, \xi_N)\big)$ is the stabilizer of $\bxi = (\xi_1, \ldots, \xi_N)$ under this action. Hence $|S| = \big[{\rm Gal}\big(\overline{\Q}/\Q\big) \, :\, {\rm Gal}\big(\overline{\Q}/\Q(\xi_1, \ldots, \xi_N)\big)\big]$ which, by Galois theory, leads to (i). Part (ii) then follows since $\chi^{\n}(\bxi) \in \Q(\xi_1, \ldots, \xi_N)$. 
\end{proof}

\begin{lemma}[cf. Lemmas 2.4 and 2.5 of \cite{DNS}] \label{Lem11_App}
 Let $\bxi \in \big(\overline{\mathbb{Q}}^{\times}\big)^N$ and consider its Galois orbit $S = \{\bxi_1, \bxi_2, \ldots, \bxi_D\}$, where $D = |S|$ and $\bxi = \bxi_1$. Write $\bxi_j = (\xi_{j,1}, \ldots, \xi_{j,N})$ for $j = 1, 2, \ldots, D$. Then 
\begin{align*}
\frac{1}{D} \sum_{j=1}^D  \sum_{\ell = 1}^N\big|\log |\xi_{j, \ell}|\big| \leq 2 h(\bxi). 
\end{align*}
\end{lemma}
\begin{proof}
We start with the case $N=1$ and simplify the notation by writing $\xi_j := \xi_{j,1}$. Let $P(x) = \sum_{j=0}^D c_jx^j = c_D \prod_{j=1}^D (x - \xi_j)$ be the minimal polynomial of $\xi = \xi_1$ over $\Z[x]$. Then 
\begin{align}\label{20241028_17:09}
\begin{split}
\frac{1}{D} \sum_{j=1}^D  \big|\log |\xi_{j}|\big| & = \frac{1}{D}  \sum_{j=1}^D  \Big( 2\log^+|\xi_{j}| - \log |\xi_{j}|\Big) = \frac{1}{D} \left(  \sum_{j=1}^D 2 \log^+ |\xi_{j}| \right) \!+ \frac{1}{D}\log \left| \frac{c_D}{c_0}\right|\\
& \leq \frac{2}{D} \left( \log|C_D| +  \sum_{j=1}^D  \log^+ |\xi_{j}| \right) = 2h(\xi).
\end{split}
\end{align}
We now consider the case $N \geq 2$. For each $\ell = 1, \ldots, N$ and each $j = 1,\ldots, D$, the numbers $\xi_{1, \ell}$ and $\xi_{j,\ell}$ are conjugates. Denote by $S_{\ell}$ the Galois orbit of $\xi_{1, \ell}$. By Lemma \ref{Lem10} we have $D = k_{\ell}|S_{\ell}|$, for a positive integer $k_{\ell}$ that is exactly the number of times that each element in the orbit $S_{\ell}$ is repeated in $\{\xi_{1, \ell}, \ldots, \xi_{D, \ell}\}$. Hence, using \eqref{20241028_17:09}, we get
\begin{align*}
\frac{1}{D} \sum_{j=1}^D  \sum_{\ell = 1}^N\big|\log |\xi_{j, \ell}|\big| =  \sum_{\ell = 1}^N \frac{1}{k_{\ell}|S_{\ell}|}\sum_{j=1}^D \big|\log |\xi_{j, \ell}|\big| = \sum_{\ell = 1}^N \frac{1}{|S_{\ell}|} \sum_{\alpha \in S_{\ell}} \big|\log|\alpha|\big| \leq \sum_{\ell = 1}^N 2h(\xi_{1, \ell}) = 2h(\bxi).
\end{align*}
\end{proof}

For $\bxi = (\xi_1, \xi_2, \ldots, \xi_N) \in \big(\overline{\mathbb{Q}}^{\times}\big)^N$, with respect to the generalized degree $\mc{D}(\bxi)$, observe from the definition that $\mc{D}(\bxi) \leq \min\{{\rm deg}(\xi_1), \ldots, {\rm deg}(\xi_N)\}$. In particular $\mc{D}(\bxi)$ can be effectively computed by considering only $\n \in \Z^N\setminus\{\0\}$ such that $\|\n\|_1 \leq \min\{{\rm deg}(\xi_1), \ldots, {\rm deg}(\xi_N)\}$.

\smallskip

In dimension $N=1$, if $\{\xi_k\}_{k \geq 1} \subset \overline{\mathbb{Q}}^{\times}$ is a strict sequence with $\lim_{k \to \infty}h(\xi_k) = 0$, one can plainly see that $\lim_{k \to \infty}{\rm deg}(\xi_k) = \infty$. This follows from the Northcott property, that states that there only finitely many algebraic numbers with bounded degree and height. Hence, if ${\rm deg}(\xi_{k_j}) \leq c$ along a subsequence, we would be led to the conclusion that infinitely many of such $\xi_{k_j}$'s are roots of unity, contradicting the assumption that the sequence is strict. The next lemma presents the corresponding claim for the generalized degree. 

\begin{lemma}[cf. Lemma 2.8 of \cite{DNS}] \label{Lem12} Let $\{\bxi_k\}_{k \geq 1} \subset \big(\overline{\mathbb{Q}}^{\times}\big)^N$ be a strict sequence with $\lim_{k \to \infty} h(\bxi_k) = 0$. Then $\lim_{k \to \infty} \mc{D}(\bxi_k) = \infty$.
\end{lemma}
\begin{proof}
Since  $\{\bxi_k\}_{k \geq 1}$ is strict, the sequence $\{\chi^{\n}(\bxi_k)\}_{k \geq 1}$ is a strict sequence in $\mathbb{Q}^{\times}$ for every $\n \in \Z^N\setminus\{\0\}$. Write $\bxi_k = (\xi_{k,1}, \ldots, \xi_{k,N})$ and let $\n = (n_1, \ldots, n_N) \in \Z^N \setminus\{\0\}$. Then 
\begin{align*}
h\big(\chi^{\n}(\bxi_k)\big) & = h\big(\xi_{k,1}^{n_1}\ldots\xi_{k,N}^{n_N}\big) \leq h\big(\xi_{k,1}^{n_1}\big) + \ldots + h\big(\xi_{k,N}^{n_N}\big) \\
& = |n_1|h(\xi_{k,1}) + \ldots + |n_N|h(\xi_{k,N})  \leq \|\n\|_{\infty} h(\bxi_k) \to 0.
\end{align*}
Hence, $\lim_{k \to \infty} {\rm deg}\big(\chi^{\n}(\bxi_k)\big) = \infty$.
Recall that, for each $k \geq 1$, there exists $\n_k \in \Z^N \setminus\{\0\}$ such that $\mc{D}(\bxi_k) = \|\n_k\|_1 \,{\rm deg}\big(\chi^{\n_k}(\bxi_k)\big)$. Arguing by contradiction, one is plainly led to the fact that $\lim_{k \to \infty} \mc{D}(\bxi_k) = \infty$.
\end{proof}

\section{Proof of Bilu's equidistribution theorem} \label{App_C}
It is possible to prove Theorem \ref{BiluThm} via abstract functional analysis machinery starting from a dense subclass of functions like in Corollary \ref{Cor3}; see \cite[\S 3.2]{P} for details. Here we give a direct proof. 

\smallskip

Let $\{\bxi_k\}_{k \geq 1} \subset \big(\overline{\mathbb{Q}}^{\times}\big)^N$ be a strict sequence with $\lim_{k \to \infty} h(\bxi_k) = 0$. We have already noted in Lemma \ref{Lem12} that $\lim_{k \to \infty} \mc{D}(\bxi_k) = \infty$, and therefore $\lim_{k \to \infty} h_{\mc{D}}(\bxi_k) = 0$. Let $S_k$ be the Galois orbit of $\bxi_k$. Let $F: (\C^{\times})^N \to \C$ be a bounded and continuous function, and let $\varepsilon >0$ be given. 

\smallskip

Choose $G: (\C^{\times})^N \to \C$ smooth and of compact support such that $|F(\z) - G(\z)| \leq \frac{\varepsilon}{2}$ in $\Gamma_{1} = \big\{ \z = (z_1, \ldots, z_N) \in (\C^{\times})^N \ : \ \sum_{j=1}^N |\log|z_j|| \leq 1 \big\}$. Since $\mu_{S_k}$ and $\mu_{(\mathbb{S}^1)^N}$ are probability measures,
\begin{align}\label{20251008_13:33}
\left|\int_{\Gamma_{1}} (F- G) \,\d \mu_{S_k} - \int_{\Gamma_{1} } (F-G) \,\d \mu_{(\mathbb{S}^1)^N} \right| \leq \left|\int_{\Gamma_{1}} (F- G) \,\d \mu_{S_k}\right| + \left| \int_{\Gamma_{1} } (F-G) \,\d \mu_{(\mathbb{S}^1)^N} \right| \leq \varepsilon\,,
\end{align}
and using $\eqref{20241025_12:40}$ with $\delta =1$ we get 
\begin{align}\label{20251008_13:34}
\begin{split}
& \left|\int_{(\C^{\times})^N \setminus \Gamma_{1}} (F- G) \,\d \mu_{S_k} - \int_{(\C^{\times})^N \setminus \Gamma_{1} } (F-G) \,\d \mu_{(\mathbb{S}^1)^N} \right|  = \left|\int_{(\C^{\times})^N \setminus \Gamma_{1}} (F- G) \,\d \mu_{S_k} \right| \\
& \qquad \qquad \qquad  \leq   \frac{\big(\|F\|_{\infty} + \|G\|_{\infty}\big)}{|S_k|}\sum_{\balpha \in S_k \setminus \Gamma_{1}} \!\! 1 \ \leq \ 2 \,\big(\|F\|_{\infty} + \|G\|_{\infty}\big) \,h(\bxi_k).
\end{split}
\end{align}
From \eqref{20251008_13:33} and \eqref{20251008_13:34} we plainly get, via the triangle inequality, 
\begin{align}\label{20251008_13:43}
\E(F - G, \bxi_k) \leq  \varepsilon + 2 \,\big(\|F\|_{\infty} + \|G\|_{\infty}\big) \,h(\bxi_k).
\end{align}
From \eqref{20251008_13:43} and our effective estimates for the smooth test function $G$ we conclude that 
\begin{align*}
\limsup_{k \to \infty} \, \E(F, \bxi_k) \leq \limsup_{k \to \infty} \, \E(F- G, \bxi_k)  + \limsup_{k \to \infty} \,\E(G, \bxi_k) \leq  \varepsilon.
\end{align*}
As $\varepsilon >0$ was arbitrary, the theorem follows.

\section*{Acknowledgements}
The second author would like to thank the DST - Government of India for the support under the DST-INSPIRE Faculty Scheme with Faculty Reg. No. IFA21-MA 168. The authors are thankful to Umberto Zannier and to the anonymous referee for the helpful comments.

\end{document}